\documentclass[10pt,a4paper]{amsart}
\usepackage{amsfonts}
\usepackage{amsthm}
\usepackage{amsmath}
\usepackage{amscd}
\usepackage[latin2]{inputenc}
\usepackage{t1enc}
\usepackage[mathscr]{eucal}
\usepackage{indentfirst}
\usepackage{graphicx}
\usepackage{graphics}
\usepackage{pict2e}
\usepackage{epic}
\usepackage[normalem]{ulem}
\numberwithin{equation}{section}
\usepackage{epstopdf}
\usepackage{verbatim}
\usepackage{amssymb}
\usepackage{tikz-cd}

\usepackage{hyperref}
\AtBeginDocument{}

\theoremstyle{plain}
\newtheorem{Th}{Theorem}[section]

\newtheorem{Cor}[Th]{Corollary}
\newtheorem{Prop}[Th]{Proposition}

\theoremstyle{definition}
\newtheorem{Def}[Th]{Definition}

\newtheorem{Rem}[Th]{Remark}
\newtheorem{?}[Th]{Problem}

\def\w{\wedge}
\def\R{\mathbb{R}}
\def\C{\mathbb{C}}

\def\om{\omega}
\def\Om{\Omega}
\def\vp{\varphi}

\def\ip{\raise1pt\hbox{\large$\lrcorner$}\>}

\DeclareMathOperator{\vol}{vol}

\begin{document}
	
\title{Deformed $G_2$-instantons on $\mathbb{R}^4 \times S^3$}

\author[U. Fowdar]{Udhav Fowdar}
	
\address{IMECC - Unicamp \\ 
Rua S\'ergio Buarque de Holanda, 651\\
13083-859\\
Campinas - SP\\
Brazil}
	
\email{udhav@unicamp.br}
		
	
\begin{abstract} 
We construct explicit examples of deformed $G_2$-instantons, also called Donaldson-Thomas connections, on $\mathbb{R}^4 \times S^3$ endowed with the torsion free $G_2$-structure found by Brandhuber et al. and on $\mathbb{R}^+\times S^3 \times S^3$ endowed with the Bryant-Salamon conical $G_2$-structure. These are the first such non-trivial examples on a $G_2$ manifold. As a by-product of our investigation we also find an associative foliation of $\mathbb{R}^4\times S^3$ by $\mathbb{R}^2 \times S^1$.
\end{abstract}

\maketitle
\tableofcontents
\section{Introduction}
A $G_2$-structure on a $7$-manifold $M$ is determined by a $3$-form $\vp$ 
which can be identified at each point with the model $3$-form 
\[\vp_0=dx_{123}+dx_{145}+dx_{167}+dx_{246}-dx_{257}-dx_{347}-dx_{356}\]
on $\R^7$. Such $\vp$ determines both a Riemannian metric $g_\vp$ and an orientation $\vol_\vp$ on $M$, and hence a Hodge star operator $*_\vp$. We call $(M, g_\vp,\vp)$ a \textit{torsion free} $G_2$ manifold if $\nabla^{g_\vp} \vp=0$, or equivalently, $d \vp =0=d*_\vp\vp$ cf. \cite{KarigiannisLeungLotay2020}. 

In many ways the geometry of $G_2$ manifolds is similar to that of Calabi-Yau $3$-folds and this has led to many conjectures originating from mirror symmetry for Calabi-Yau $3$-folds to be formulated for $G_2$ manifolds as well. 
Based on ideas originating from the SYZ conjecture, in \cite{LeeLeung2009} the authors introduced the following notion of deformed $G_2$-instantons: 
\begin{Def}
A connection $1$-form $A$ on $(M,g_\vp,\vp)$ is called a $G_2$-instanton if its curvature $2$-form $F_A:=dA+\frac{1}{2}[A\w A]$ satisfies
\begin{equation}
	F_A \w *_\vp\vp =0. \label{g2instantonequ}
\end{equation}
If instead, $F_A$ satisfies
\begin{equation}
	\frac{1}{6}F_A^3-F_A \w *_\vp\vp =0,\label{dG2instantonequ}
\end{equation}
then $A$ is called a deformed $G_2$-instanton, sometimes also called deformed Donaldson-Thomas connection cf. \cite{Karigiannis2009, LeeLeung2009, Lotay2020}.
\end{Def}
While the definition of $G_2$-instantons arises naturally as a higher dimensional analogue of anti-self-dual instantons on $4$-manifolds \cite{Carrion1998}, the motivation behind the definition of deformed $G_2$-instantons comes from the fact that they arise as mirrors, via a real Fourier-Mukai transform, to certain (co)associative cycles cf.  \cite{LeeLeung2009, Leung2000}. These connections can be considered as the $G_2$ analogues of Hermitian Yang-Mills and deformed Hermitian Yang-Mills connections on Calabi-Yau $3$-folds which, by contrast to their $G_2$ counterparts, have been studied extensively cf. \cite{Chen2021, CollinsXieYau2018, Leung2000, Marino2000}.

Just like $G_2$-instantons correspond to critical points of the Yang-Mills functional, it was shown in \cite[Theorem 5.13]{Karigiannis2009} that deformed $G_2$-instantons arise as critical points of a natural Chern-Simons type functional, see (\ref{chernsimonsfunctional}) below. 
Motivated by the work in \cite{LeeLeung2009} one is mainly interested in finding solutions to (\ref{dG2instantonequ}) in the case when $A$ is a unitary connection on a complex line bundle. So in this article we shall only be concerned with the case when the gauge group $G=U(1)$ so that $F_A=dA$. Furthermore, we shall identify the Lie algebra $\mathfrak{u}(1)$ with $\R$ and view $F_A$ as a \textit{real} $2$-form defining an integral cohomology class in $H^2(M^7,\mathbb{Z})$. 

The existence of the first non-trivial examples of deformed $G_2$-instantons was demonstrated in \cite{Lotay2020} on certain \textit{nearly parallel} $G_2$ manifolds i.e. when the $3$-form $\vp$ satisfies  $d\vp=*_\vp\vp$. In this article we construct the first non-trivial solutions to (\ref{dG2instantonequ}) on a \textit{torsion free} $G_2$ manifold, see Theorem \ref{maintheorem} and \ref{maintheoremtwo}. Here `non-trivial' means that the connections are neither flat nor arise as pullback from lower dimensional constructions. 

The manifold that we are interested in is $\R^4 \times S^3$ which is known to admit a $1$-parameter family of $G_2$ metrics cf. \cite{Bazaikin2013, Foscolo2018}. In this family there are two examples which are explicitly known, namely the Bryant-Salamon one \cite{Bryant1989} and the Brandhuber-Gomis-Gubser-Gukov (BGGG) one \cite{Brandhuber01}. We shall construct explicit $1$-parameter families of deformed $G_2$-instantons for the BGGG $G_2$-structure and the conical Bryant-Salamon $G_2$-structure.
Note that since $H^2(S^3,\mathbb{Z})=\{0\}$ the instantons that we construct will be on the trivial line bundle. Due to the close relation between deformed $G_2$-instantons and calibrated submanifolds we are also led to consider the latter in $\R^4  \times S^3$. This leads us to find an associative foliation of $\R^4\times S^3$ by $\R^2 \times S^1$, see Section \ref{calibratedsubmanifolds}. We hope that the results of this paper provide motivation to initiate a search for deformed $G_2$-instantons on \textit{compact} $G_2$ manifolds. \\

\noindent\textbf{Acknowledgements.}
The author would like to thank Jason Lotay for bringing the question addressed in this article to his attention, and also Daniel Fadel for several helpful comments. 
This work was supported by the S\~ao Paulo Research Foundation (FAPESP) [2021/07249-0].

\section{The $SU(2)^2\times U(1)$-invariant $G_2$-structures}\label{theg2manifoldsection}

We begin by describing the explicitly known $G_2$-structures on $\R^4 \times S^3$ for which we shall construct deformed $G_2$-instantons. The examples that we are interested in all admit a cohomogeneity one action of $SU(2)^2\times U(1)$.  
To ease the comparison of our deformed $G_2$-instantons with the $G_2$-instantons constructed in \cite{Lotay2016} we shall adopt the same notation as therein.

First consider the Lie algebra $\mathfrak{su}_2$ with its standard basis 
\[T_1= 
\begin{pmatrix}
	i & 0\\
	0&-i
\end{pmatrix}, \ \
T_2= 
\begin{pmatrix}
	0 & 1\\
	-1&0
\end{pmatrix}, \ \
T_3= 
\begin{pmatrix}
	0 & i\\
	i&0
\end{pmatrix},\]
such that $[T_i,T_j]=2 \varepsilon_{ijk}T_k$. Denoting by $\sigma_i$ the dual basis, we have that 
$$ d\sigma_i=-\varepsilon_{ijk}\sigma_j \w \sigma_k.$$ 
By identifying the tangent spaces of $S^3 \times S^3$ with
$\mathfrak{su}_2 \oplus \mathfrak{su}_2$ we can define a basis of the tangent bundle by $T^+_i=(T_i, T_i)$ and $T^-_i=(T_i,-T_i)$. The vectors fields $\{T^+_i\}_{i=1}^3$ define a Lie subalgebra and as such correspond to the tangent space of a diagonal $S^3$ embedded in $S^3 \times S^3$. If we now denote by $\eta_{i}^{\pm}$ the corresponding dual $1$-forms, then the Maurer-Cartan relations give
\begin{align*}
	d \eta^+_i &= -\varepsilon_{ijk} (\eta^{++}_{jk} + \eta_{jk}^{--}),\\
	d \eta^-_i &= -2\varepsilon_{ijk} \eta^{-+}_{jk},
\end{align*}
where we follow the convention that adjacent $1$-forms are wedged together. Equipped with these forms we can define $1$-parameter families of $SU(2)^2\times U(1)$ invariant $SU(3)$-structures on $S^3 \times S^3$ by
\begin{gather}
	h = (2A_1)^2 (\eta_{1}^+ \otimes  \eta_{1}^+) + (2A_2)^2( \eta_{2}^+ \otimes  \eta_{2}^+ + \eta_{3}^+ \otimes  \eta_{3}^+)\ +\label{anstazmetric}\\
	\ \ \ \ (2B_1)^2 (\eta_{1}^- \otimes \eta_{1}^-) + (2B_2)^2 (\eta_{2}^- \otimes \eta_{2}^- + \eta_{3}^- \otimes \eta_{3}^-),\nonumber\\[2pt]
	\om = 4 A_1 B_1 (\eta_{11}^{-+}) + 4 A_2 B_2 (\eta_{22}^{-+} + \eta_{33}^{-+}),\label{anstazform}\\[2pt]
	\Om^+ = 8 B_1 B_2^2 \eta_{123}^{---} - 8 A_2^2 B_1 \eta_{123}^{-++}- 8 A_1 A_2 B_2 (\eta_{123}^{++-} + \eta_{123}^{+-+}),\label{anstazomplus}\\[2pt]
	\Om^- = 8A_1 B_2^2 \eta_{123}^{+--} -8A^2_2 A_1 \eta_{123}^{+++} + 8 B_1 B_2 A_2 (\eta_{123}^{--+} +  \eta_{123}^{-+-}),\label{anstazomminus}
\end{gather}
where $A_i,B_i : \R^+_t \to \R$ for $i=1,2.$ 
\begin{Rem}
Note that the $\{1\} \times U(1) \subset SU(2)^2 \times U(1)$ action is generated by the vector field $T^+_1$. This can be seen in (\ref{anstazmetric}) from the fact that coefficient functions of $\eta_{2}^{\pm}$  and $\eta_{3}^{\pm}$ are equal. In the Bryant-Salamon case there is in fact an $SU(2)^3$ symmetry since all three vector fields $T^+_i$ are Killing as we shall see next.
\end{Rem}
We can now define an $SU(2)^2\times U(1)$-invariant $G_2$-structure on $\R^+_s \times S^3 \times S^3$ by the $3$-form
\begin{equation}
\vp=dt \w \om(t) + \Om^+(t),\label{the3form}
\end{equation}
which in turn defines the data
\begin{gather}
g_{\vp}=dt^2 + h(t),\\
*_\vp\vp= \frac{1}{2} \om(t) \w \om(t) - dt \w \Om^-(t).
\end{gather}

The Bryant-Salamon $G_2$-structure constructed in \cite{Bryant1989} is obtained as follows. We define a new radial coordinate $r \in [c,\infty)$ by $$t=\int_{c}^r\frac{1}{\sqrt{1-c^3x^{-3}}}dx,$$ 
and set
\begin{equation*}
	A_i=\frac{r}{3}\sqrt{1-c^3r^{-3}},\ \ \
	B_i=\frac{r}{\sqrt{3}},
\end{equation*}
where $c$ is a non-negative constant. If $c=0$ then we get a conical $G_2$-structure and if $c \neq 0$ then without loss of generality we can set $c=1$ to get a complete metric on $\R^4 \times S^3$. Observe that as $r$ tends to infinity, all the coefficient functions $A_i$ and $B_i$ grow linearly in $r$ as the cone. As such the complete Bryant-Salamon metric is said to be of type AC i.e. asymptotically conical. 

The BGGG $G_2$-structure constructed in \cite{Brandhuber01} is obtained as follows. We define a new radial coordinate $r \in [9/4,\infty)$ by $$t=\int_{9/4}^r\frac{\sqrt{(x-3/4)(x+3/4)}}{(x-9/4)(x+9/4)}dx,$$ 
and set
\begin{align*}
	A_1=\frac{\sqrt{(r-9/4)(r+9/4)}}{\sqrt{(r-3/4)(r+3/4)}},\ \
	A_2=\sqrt{\frac{(r-9/4)(r+3/4)}{3}},\\
	B_1=\frac{2r}{3},\ \
	B_2=\sqrt{\frac{(r-9/4)(r+3/4)}{3}}.
\end{align*}
Observe that as $r$ tends to infinity, $A_1$ tends to $1$ while the other coefficients grow linearly in $r$. Such a $G_2$ metric is said to be of type ALC (asymptotically locally conical) i.e. it is asymptotic to a metric on an $S^1$ bundle over a $6$ dimensional Calabi-Yau cone over a Sasaki-Einstein $5$-manifold. In this case the latter $5$-manifold is just $S^2 \times S^3$. 
On the other hand as $r \to 9/4$, $A_i$ tend to $0$ and hence the diagonal $S^3$ fibres collapse just in the Bryant-Salamon case. The metric $g_\vp$ on $\R_s \times S^3 \times S^3$ then extends smoothly to $S^3=SU(2)  \times SU(2)/\Delta SU(2)$, viewed as the zero section, to give a complete metric on $\R^4 \times S^3$, viewed as the spinor bundle of $S^3$.
Having now described the $G_2$-structures of interest, next we compute the instanton equations.

\section{The instanton ODEs}
In \cite{Lotay2016} Lotay-Oliveira showed that an $SU(2)^2$-invariant $G_2$-instanton on $\R^4 \times S^3$, endowed with the $G_2$-structures described in the previous section, is necessarily of the form
\begin{equation}
	A=f_1(r)\eta_{1}^++f_2(r)\eta_{2}^++f_3(r)\eta_{3}^+.\label{generalconnectionansatz}
\end{equation}
Motivated by this we shall restrict our investigation to connections of the latter form.
We begin by computing the fundamental system of ODEs defining $G_2$ and deformed $G_2$-instantons. 

\begin{Prop}
Consider a connection $1$-form $A$ of the form (\ref{generalconnectionansatz})
on $\R^4 \times S^3$ endowed with the Bryant-Salamon $G_2$-structure.
Then the $G_2$-instanton equation (\ref{g2instantonequ}) becomes equivalent to
\begin{equation}
	(r^4-c^3r)f'_i(r)-(2r^3+c^3)f_i(r)=0, \label{g2odecone}
\end{equation}
where $i=1,2,3$ and the deformed $G_2$-instanton equation (\ref{dG2instantonequ}) becomes equivalent to
\begin{align}
	(r^4-c^3r+\frac{27}{4}f_i(r)^2)f'_i(r)\ -&\label{dg2odecone}\\
	(2r^3+c^3-\frac{27}{4}(f_j(r)f'_j(r)+f_k(r)f'_k(r))) f_i(r)& = 0, \nonumber
\end{align}
where $(i,j,k)=(1,2,3)$ + cyclic permutation.
\end{Prop}

\begin{Prop}
Consider a connection $1$-form $A$ of the form (\ref{generalconnectionansatz})
on $\R^4 \times S^3$ endowed with the BGGG $G_2$-structure.
Then the $G_2$-instanton equation (\ref{g2instantonequ}) becomes equivalent to
\begin{gather}
	(16r^2-81)(16r^2-9)f'_1(r)-2304rf_1(r)=0,\label{g2instantonODE}\\
	r(16r^2-81)f'_i(r)-(4r+3)(4r^2-9r+\frac{27}{2})f_i(r)=0,
\end{gather}
where $i=2,3$ and the deformed $G_2$-instanton equation (\ref{dG2instantonequ}) becomes equivalent to
\begin{align}
	\big((16r^2-81)(16r^2-9)+576f_1(r)^2\big)f_1'(r)\ + &\label{dG2odegeneralBGGG}\\
	576(f_2(r)f_2'(r)+f_3(r)f_3'(r)-4r)f_1(r)& = 0,\nonumber
\end{align}
\begin{align}
	(16r^3-81r+18f_i(r)^2)f_i'(r)\ -&\\
	((4r+3)(4r^2-9r+\frac{27}{2})-18(f_1'(r)f_1(r)+f_j'(r)f_j(r)))f_i(r)& = 0,\nonumber
\end{align}
where $(i,j)=(2,3)$ + cyclic permutation.
\end{Prop}
In all the above cases the $G_2$-instanton ODEs are easily solved to give explicit solutions cf. \cite[Corollary 1]{Lotay2016}. While the $G_2$-instanton equation (\ref{g2instanton}) is a linear condition (since the gauge group is abelian), the deformed condition (\ref{dG2instantonequ}) is non-linear thereby making it much harder to solve in general. Furthermore, not all solutions to the above ODEs will lead to well-defined instantons on $\R^4 \times S^3$. The reason for this is that not all $1$-forms of the form (\ref{generalconnectionansatz}) are globally well-defined on $\R^4 \times S^3$. 
From \cite[Lemma 9]{Lotay2016} we require that the coefficient functions $f_i$ be even in the variable $s$ and that they vanish at the zero section $S^3$, i.e. at $t=0$, for the connection form $A$ to extend smoothly across the zero section. Using the latter as an initial condition leads to a singular initial value problem and as such standard methods from the theory of ODEs are not usually applicable. 
Thus, we have not been able to analyse the general case, but we have nonetheless been able to find certain explicit family of solutions which we describe next.

\section{Deformed $G_2$-instantons on the BGGG manifold}
We construct deformed $G_2$-instantons on $\R^4 \times S^3$ endowed with the BGGG  $G_2$-structure described in Section \ref{theg2manifoldsection}. We shall consider the simplest case when the coefficient functions $f_2,f_3$ are both zero so that
\begin{equation}
A=f(r) \eta_{1}^+.\label{ansatzconnection}
\end{equation}
Note that $A$ can be characterised as the unique $SU(2)^2\times U(1)$-invariant connection arising from the ansatz (\ref{generalconnectionansatz}). By Schur's lemma there is only one other $SU(2)^2\times U(1)$-invariant $1$-form on $\R^4 \times S^3$ namely $\eta_{1}^-$ but from the results in \cite{Lotay2016} we already know that this does not give rise to a smooth $G_2$-instanton. Observe that with the above ansatz we no longer have a system of ODEs for the deformed condition but just a single one, which is nonetheless still non-linear.
We begin by recalling the $G_2$-instanton arising from ansatz (\ref{ansatzconnection}).
\subsection{Example of a $G_2$-instanton}
With $A$ as above one easily solves (\ref{g2instantonODE}) to get 
\begin{equation}
A=c_0\frac{16 r^2-81}{16 r^2 -9}\eta_{1}^+,\label{g2instanton}
\end{equation}
where $c_0\in \R$. Observe that this corresponds to the harmonic $1$-form dual to the Killing vector field $T^+_1$ i.e. 
\[A = c_0 (T^+_1)^\flat = c_0 A_1^2 \eta_{1}^+.\] 
Indeed we have that for any Killing vector field $X$ on a $G_2$ manifold:
\[2d X^\flat \w *_\vp\vp=d( (X \ip \vp)\w \vp)=\mathcal{L}_{X}\vp \w \vp =0.\]
It is also straightforward to show that $X^\flat$ is always a harmonic $1$-form i.e. the connection form is in Coulomb gauge. 
As we shall see below there is rather close relation between the above solution and solutions to the deformed equation.
\subsection{Examples of deformed $G_2$-instantons}
We now consider the deformed $G_2$-instanton equation. Again with the above ansatz we find that (\ref{dG2odegeneralBGGG}) simplifies to
\begin{equation}
\big((16r^2-81)(16r^2-9)+576f(r)^2\big)f'(r)-2304rf(r)=0.\label{dG2ode}
\end{equation}
Comparing with (\ref{g2instantonODE}), we now see that the presence of the term involving $f(r)^2$ in (\ref{dG2ode}) makes the ODE non-linear. A general solution is implicitly given by
\begin{equation}
	24 \tan(f(r)/3+c) f(r) = 16r^2-81,\label{solution}
\end{equation}
where $c \in \R$. The reader can easily verify this by differentiating (\ref{solution}). 

Note that making the substitution $z=r-9/4$ in (\ref{dG2ode}) and looking for a Taylor series solution $p(z)=f(r)$ around $z=0$ shows that either $p(z)=0$ or
\[p(z)=a+\frac{9}{a}z+\frac{2a^2-81}{a^3}z^2-\frac{9(7a^2-162)}{a^5}z^3-\frac{22a^4-1944a^2+32805}{a^7}z^4+\cdots\]
where $a\in \R$. In particular, we see that such a solution cannot vanish at $z=0$ i.e. $r=9/4$, and hence from \cite[Lemma 9]{Lotay2016} we deduce that this does not give rise to a smooth connection $1$-form on $\R^4\times S^3$ since this $1$-form will have a singularity at the zero section $S^3$. Nonetheless, we shall see below that there are solutions defined by (\ref{solution}) which vanish at $r=9/4$.
\begin{Rem}
Note that if we require that $f(9/4)=0$ then (\ref{dG2ode}) corresponds to a singular initial value problem and as such one does not have uniqueness of solution for such an initial data. 
This is also the case for the $G_2$-instanton equation (\ref{g2instantonODE}) and indeed we see that for any $c_0\in\R$ in (\ref{g2instanton}) we get a solution which vanishes at $r=9/4$. This phenomenon occurs whenever one has a singular orbit in a cohomogeneity one manifold cf. \cite{Eschenburg2000}. Since the point $r=9/4$ is a \textit{regular} singular point for (\ref{g2instantonODE}), one can apply Malgrange's method \cite{Malgrange74} and still get local existence and uniqueness results, see \cite{Foscolo2018, Lotay2016} for such applications. However, the existence of the multi-valued solution (\ref{solution}) and above expression for the power series solution $p(z)$ suggest that the latter technique might not be applicable to study  the deformed instanton equations.
\end{Rem} 

Note that the unlike the (abelian) $G_2$-instanton equation, the deformed $G_2$-instanton equation is not invariant under scaling. However, the latter does admit a $\mathbb{Z}_2$-symmetry i.e. if $A$ is a deformed $G_2$-instanton then so is $-A$. Indeed we see that if $f_1(r)$ is solution to (\ref{solution}) with $c=c_1$ then so is $-f_1(r)$ but for $c=-c_1$. So henceforth, without loss of generality, we shall restrict to the case when $c \geq 0$.  

Due to the periodicity of the tangent function there are infinitely many branches of solutions to (\ref{solution}). This is readily seen from the fact that
\begin{equation}
	\tan(y/3+c)=\frac{C}{y}\label{plot}
\end{equation}
has infinitely many solutions for each $C \geq 0$, see Figure 1. 
\begin{figure}
	\centering
	\includegraphics[height=7cm]{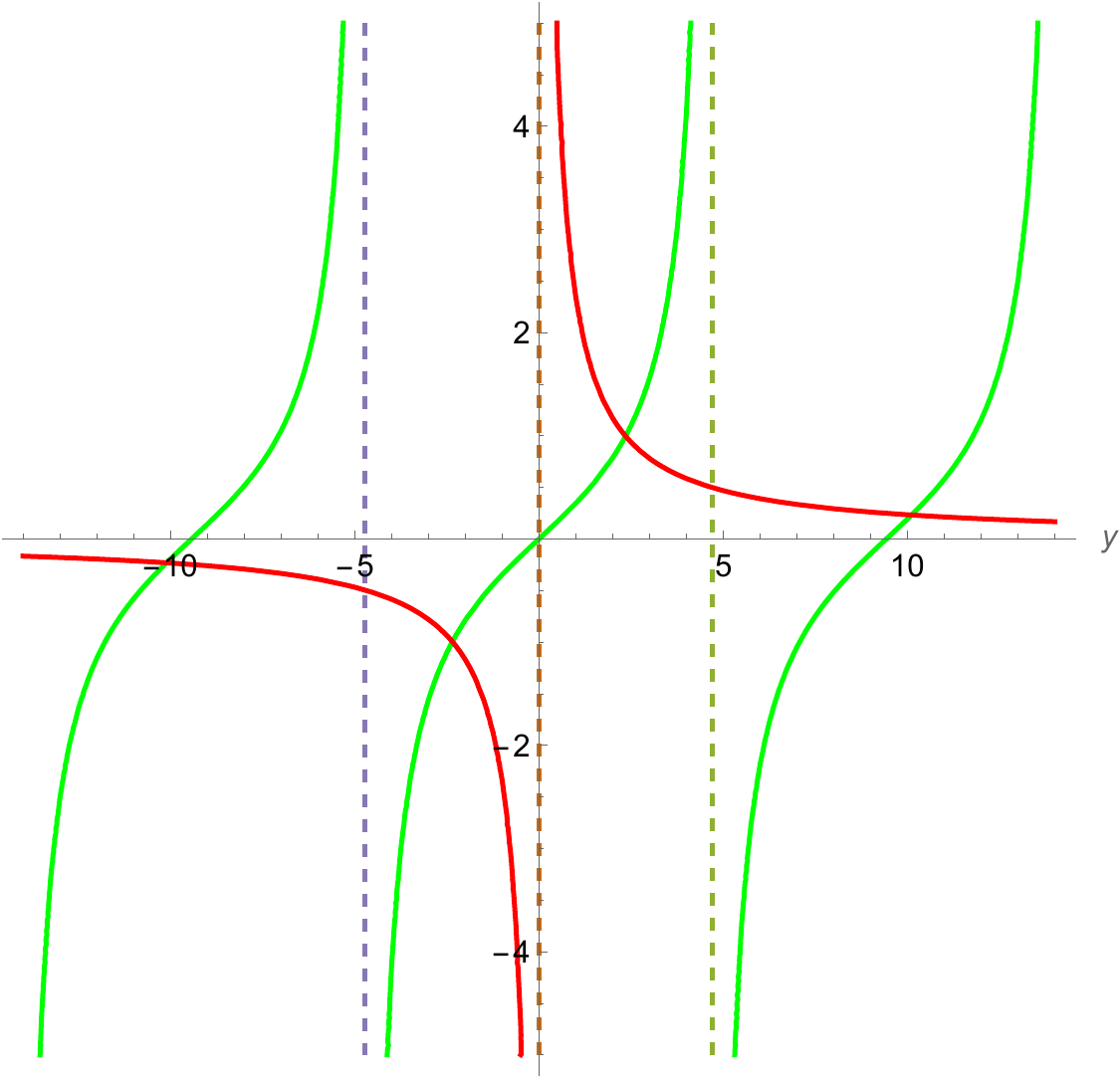}
	\caption{{Graphical solutions to (\ref{plot}) for $c=0$ and $C>0$.}}
\end{figure}
Note that since $f(r)$ is implicitly defined by (\ref{solution}) it is easier to study its behaviour by considering solutions to (\ref{plot}) as $c$ and $C$ vary. In particular, observe that all solutions to (\ref{solution}) are asymptotically constant. Moreover, for each $c\in[0,\pi/2)$ there exist a non-negative solution satisfying $f(\frac{9}{4})=0$ which is asymptotic to the line $f=\frac{3 \pi}{2}-3c$, see Figures 2 and 4, and hence this converges to the flat connection as $c\to \pi/2$. Indeed the latter is readily deduced from Figure 1 by observing that as $C \to \infty$ we have that $y \to \frac{3 \pi}{2}-3c$ and that as $C \to 0$ we have that $y \to 0$. Other branches with $f(r)$ not vanishing at $r=9/4$ are depicted in Figures 3 and 4 (these correspond to the aforementioned power series solution $p(z)$). Differentiating (\ref{solution}) and using that $f(9/4)=0$, one finds that 
\[f'(9/4)=3\cot(c).\]
This shows that for $c=0$ the solution is not smooth at $r=9/4$, as seen also in Figure 3. So we can summarise the above discussion into:
\begin{Th}\label{maintheorem}
In the BGGG case, there exists a $1$-parameter family of deformed $G_2$-instantons of the form $A_c=f(r)\eta_{1}^+$, where $f(r)$ is defined by (\ref{solution}) with $c\in (0,\pi/2)$. Moreover, as $c \to \pi/2$ the connection $A_c$ converges to the flat connection.
\end{Th}
\begin{figure}
	\centering
	\includegraphics[height=7cm]{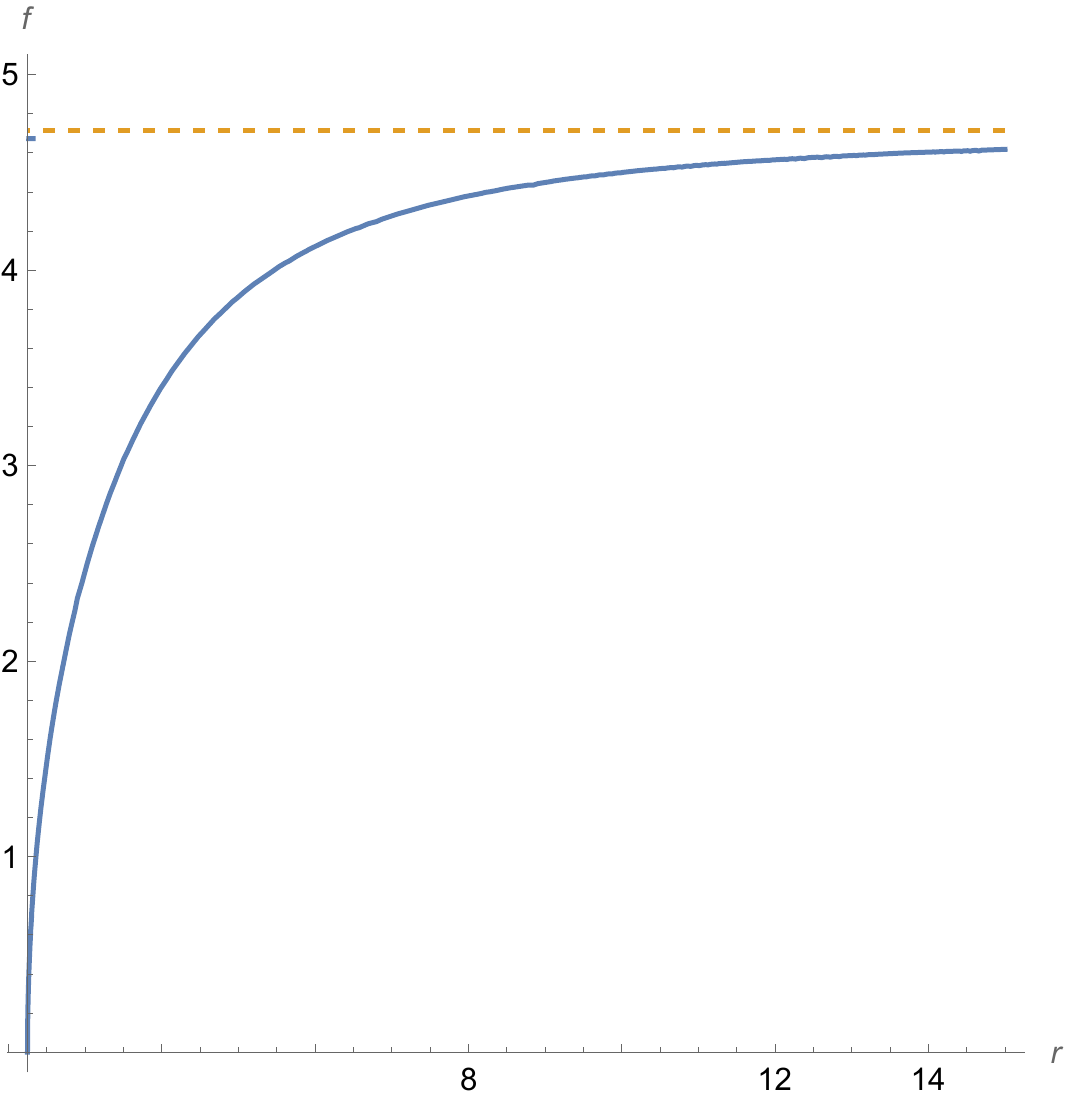}
	\caption{{A solution to (\ref{solution}) for $c=0$ asymptotic to $f=3\pi/2$.}}
\end{figure}
\begin{figure}
	\centering
	\includegraphics[height=7cm]{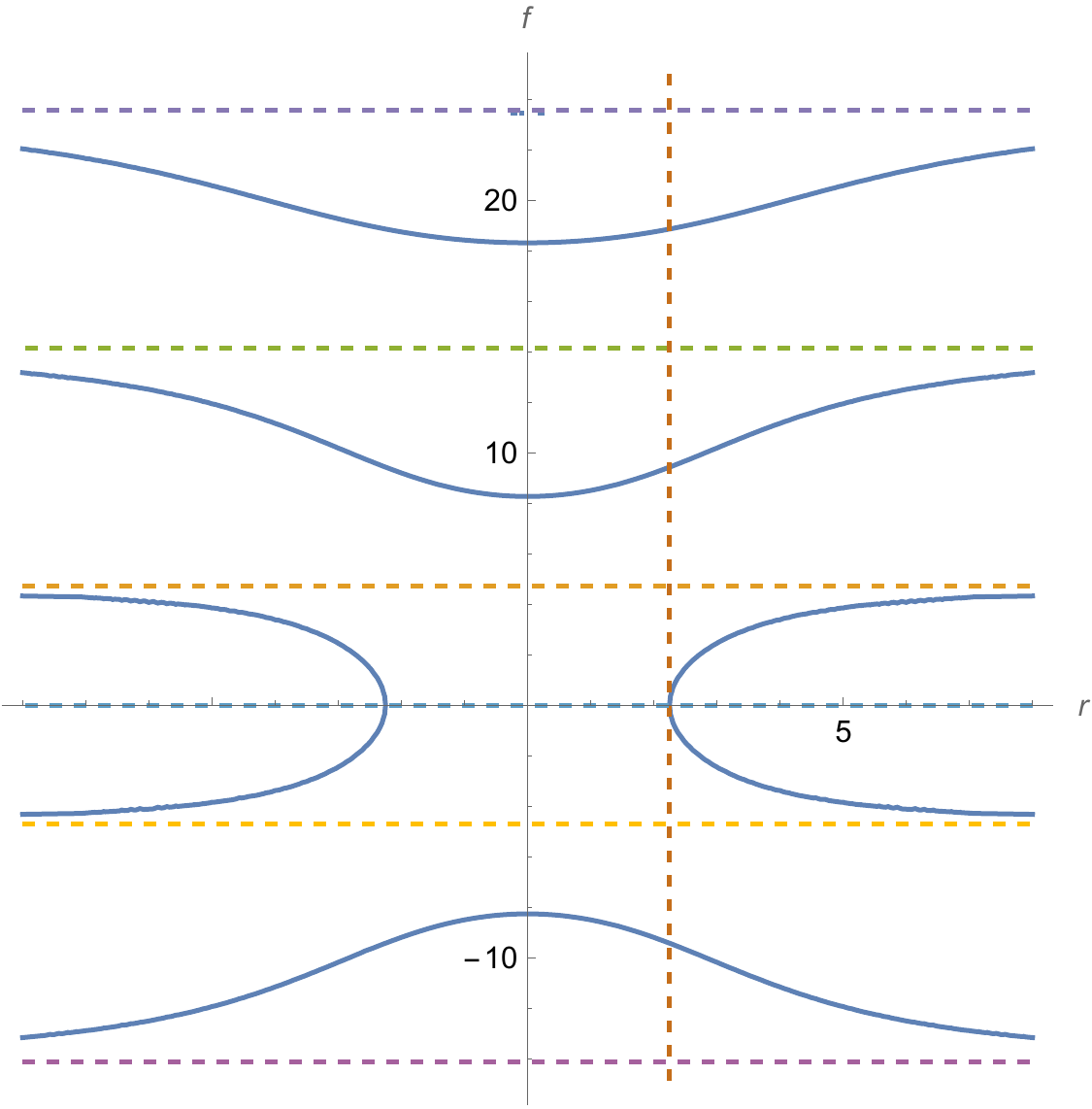}
	\caption{{Solutions to (\ref{solution}) for $c=0$ with vertical line denoting $r=9/4$.}}
\end{figure}
\begin{figure}
	\centering
	\includegraphics[height=7cm]{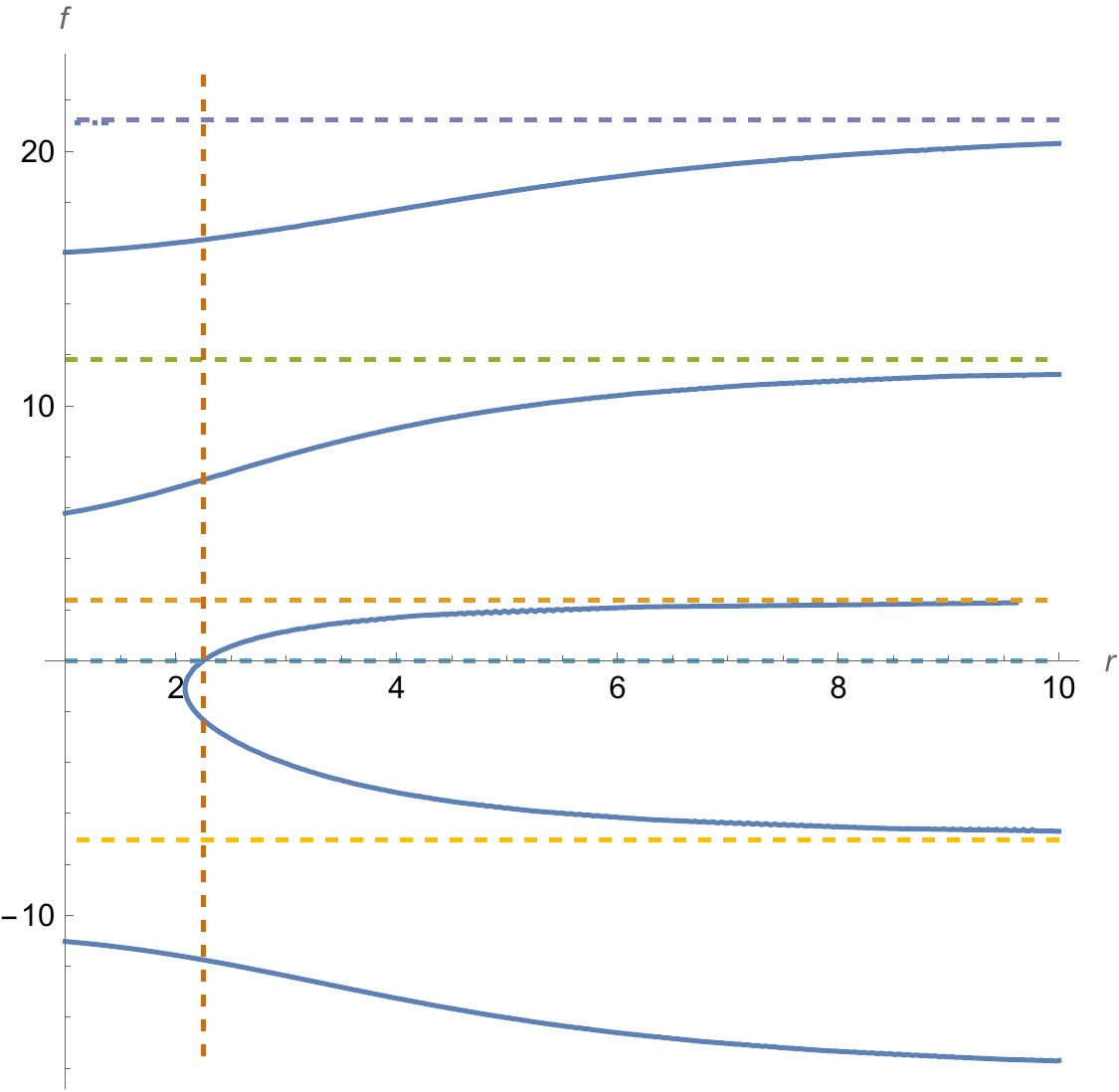}
	\caption{{Solutions to (\ref{solution}) for a $c\in(0,\pi/2)$ with vertical line denoting $r=9/4$.}}
\end{figure}
Recall that deformed $G_2$-instantons arise as critical points of a Chern-Simons type functional introduced  in \cite[(5.18)]{Karigiannis2009}, see also \cite[(5.5)]{Lotay2020}). The latter can be expressed as
\begin{equation}
	\mathcal{F}(\underline{A})=\frac{1}{2}\int_{M \times [0,1]} -F_{\underline{A}}^2 \w *_\vp\vp+\frac{1}{12} F_{\underline{A}}^4,\label{chernsimonsfunctional}
\end{equation}
where $\underline{A}:=A_0+s(A-A_0)$ is a connection $1$-form on $[0,1]_s\times M^7$ and $A_0$ is a fixed reference connection on $M^7$. A direct computation shows that $\mathcal{F}(\underline{A})=0$ for any $A$ of the form (\ref{ansatzconnection}) with $A_0=0$. Thus, we have that:
\begin{Prop}
$\mathcal{F}\equiv 0$ for the deformed $G_2$-instantons given in Theorem \ref{maintheorem}.
\end{Prop}
In fact, for any deformed $G_2$-instanton of the form (\ref{generalconnectionansatz}) one has that $\mathcal{F}\equiv 0$. This is rather different from the results found in \cite{Lotay2020} on the nearly parallel, and more generally coclosed, $G_2$ manifolds. Observe that it is rather curious that for both the $G_2$ and deformed $G_2$-instantons the function $f(r)$ is asymptotically constant. 
This is suggestive that there might be a relation between these solutions which is what we investigate next.
\subsection{A limiting behaviour}
In the large volume limit it is known that the leading order term of the deformed Hermitian  Yang-Mills equation corresponds to the usual Hermitian Yang-Mills equation cf. \cite{CollinsXieYau2018}. So one might expect a similar behaviour to hold for the deformed $G_2$-instanton equation. We describe an analogous argument but rather than considering a large volume, we shall fix the $G_2$-structure (and hence $\vol_\vp$) and instead consider a small deformed $G_2$-instanton. More precisely,
\begin{Prop}\label{propconvergencetoflat}
Suppose that $A_n$ is a sequence of smooth deformed $G_2$-instantons on $(M^7,\vp,g_\vp)$ such that as $n \to \infty$, $c_n:=\sup |F_{A_n}| \to 0$ i.e. $A_n$ converges to a flat connection. Then $B_n:=c_{n}^{-1}A_{n}$ converges smoothly, modulo gauge transformation, to a $G_2$-instanton $B_\infty$.
\end{Prop}
\begin{proof}
Since $A_n$ is a deformed $G_2$-instanton, it follows that $B_n$ satisfies
\begin{equation}
	\frac{1}{6}c_n^2 F_{B_n}^3-F_{B_n} \w *_\vp\vp=0.
\end{equation}
By definition we also have that $\sup |F_{B_n}| =1$ and hence as $n\to \infty$ we see that $c_n^2 \sup|F_{B_n}^3| \to 0$. From Uhlenbeck's results in \cite{Uhlenbeck1982} we can extract a weak limit $B_\infty$. Since $G_2$-instantons are Yang-Mills connections, it follows that the limit is also smooth cf. \cite{Tian2000}. 
\end{proof}
Note that if $M$ is compact then we can replace the sup norm by the $L^2$-norm and the above argument still holds since $L^3 \hookrightarrow L^2$. Thus, again by Uhlenbeck's results we can
extract a weak limit $B_\infty$ with finite energy. However, obtaining regularity in this case will require analogues of Tian's results in \cite{Tian2000} and one expects smooth convergence to occur outside of a blow-up set corresponding to an associative cf. \cite[Theorem 4.2.3]{Tian2000}. 

We now give a concrete illustration of Proposition \ref{propconvergencetoflat}. Recall that in Theorem \ref{maintheorem} we saw that as $c \to \pi/2$ our solution converges to a flat connection, so as a consequence of the above we show that:
\begin{Cor}
Let $A_c$ denote the deformed $G_2$-instanton as in Theorem \ref{maintheorem} then setting $c=\tan^{-1}(\epsilon^{-1})$ we have that $\epsilon^{-1} A_c$ converges to the $G_2$-instanton given by (\ref{g2instanton}) as $\epsilon \to 0$. 
\end{Cor}
\begin{proof}
Let $\epsilon$ denote a non-negative constant and set $B=\epsilon^{-1}A_c$. Then $B$ satisfies  
\begin{equation}
	\frac{1}{6}\epsilon^2 F_B^3-F_B \w *_\vp\vp =0.\label{dG2instantonequmod}
\end{equation}
Writing $B=f(r)\eta_{1}^+$, 
the general solution is given by
\begin{equation}
	24\epsilon f(r) \tan(\epsilon f(r)/3+c)  = 16r^2-81.
\end{equation}
This is of course equivalent to (\ref{solution}).
Using the trigonometric sum identities we can express the above as
\begin{equation*}
	\tan(c) \big(24 \epsilon f \cos(\epsilon f/3) +(16r^2-81)\sin(\epsilon f/3)\big)=\big((16r^2-81)\cos(\epsilon f/3)-24 \epsilon f \sin(\epsilon f/3)\big) .
\end{equation*}
From our previous results we know that $\epsilon f(r) \to 0$ as $c \to \pi/2$, and hence we have that $\sin(\epsilon f/3) \approx \epsilon f/3$ and $\cos(\epsilon f/3) \approx 1$. Thus, the above becomes
\[\frac{\epsilon}{3}\tan(c) f(r) \approx \frac{16r^2-81}{16 r^2-9}-\frac{8 (\epsilon f)^2}{16r^2-9}.\]
Setting $c=\tan^{-1}(\epsilon^{-1})$ and taking $\epsilon\to 0$ then gives the solution (\ref{g2instanton}). This concludes the proof.
\end{proof}
Since deformed $G_2$-instantons are expected to be mirrors to associative cycles, it seems natural to ask if the above convergent sequence has an analogue interpretation in terms of convergence of associatives in the mirror. 
\section{Deformed $G_2$-instantons on the Bryant-Salamon cone}

We shall now construct instantons on $\R^+ \times S^3 \times S^3$ endowed with the Bryant-Salamon conical $G_2$-structure. We consider a connection of the form 
\begin{equation}
	A=f(r)(a_1\eta_{1}^++a_2\eta_{2}^++a_3\eta_{3}^+),\label{connectionansatzBScase}
\end{equation}
where $a_i$ are constants. From (\ref{g2odecone}), with $c=0$, we have that the general solution to the $G_2$-instanton equation is given by 
\begin{equation}
	f(r)=c_0 r^2.\label{solutiong2cone}
\end{equation}
Again we see that the resulting connection $1$-form corresponds to the dual of the Killing vector fields $T_i^+$. On the other hand from (\ref{dg2odecone}) we see that the deformed equation becomes
\begin{equation}
	f'(r)(r^4+\frac{27}{4}(a_1^2+a_2^2+a_3^2)f(r)^2)-2f(r)r^3=0.\label{dg2odeconeansatz}
\end{equation}
A $1$-parameter family of solutions to (\ref{dg2odeconeansatz}) is implicitly given by
\begin{equation}
	\log(c f(r))f(r)^2=\frac{2 r^4}{27(a_1^2+a_2^2+a_3^2)},\label{solutiondg2cone}
\end{equation}
or equivalently, one can express the above solution as
\begin{equation}
	f(r)=\frac{1}{c}\exp\Big(\frac{1}{2}W\Big(\frac{4c^2r^4}{27(a_1^2+a_2^2+a_3^2)}\Big)\Big),\label{functionlog}
\end{equation}
where $W$ denotes the Lambert $W$ function i.e. the inverse of $h(x)=xe^{x}$. As in the previous section we can restrict to the case when $c>0$, so we can summarise the above into:
\begin{Th}\label{maintheoremtwo}
On the Bryant-Salamon cone, there exist deformed $G_2$-instantons of the form (\ref{connectionansatzBScase}), where $f(r)$ is defined by (\ref{functionlog}) with $c \in \R^+$ and $a_i \in \R$.
\end{Th}
From (\ref{functionlog}) we see that $f(r)$ is strictly increasing. Since $\log(x)$ decays slower than $x^\epsilon$ for any $\epsilon>0$, from (\ref{solutiondg2cone}) one has that the asymptotic behaviour of $f(r)$ is sub-quadratic i.e. $r^{2-\epsilon}$. Hence as in the previous section we see that both the $G_2$ and deformed $G_2$-instantons have rather similar asymptotic behaviour. 
It is likely that similar smooth solutions exist for the complete Bryant-Salamon $G_2$-structure but we have not been able to find explicitly such solutions.

\section{Evolution equations on hypersurfaces}

We now describe a more general framework to construct deformed $G_2$-instantons via evolution equations. 
First we recall the results of Bryant and Hitchin in \cite[Theorem 4]{BryantEmbedding} and \cite[Theorem 8]{Hitchin01stableforms} that assert that if a $6$-manifold $N$ admits a real analytic $SU(3)$-structure determined by the pair $(\om_0,\Om^+_0)$ satisfying
\begin{gather}
	d\om_0 \w \om_0 =0,\label{halfflatone}\\
	d\Om^+_0=0,\label{halfflattwo}
\end{gather}
then there exists a solution to the Hitchin's flow:
\begin{gather}
\frac{\partial \Om^+}{\partial t}=d\om,\\
\frac{\partial \om^2}{\partial t}=-2d\Om^-,	
\end{gather}
with $\om(0)=\om_0$ and $\Om^+(0)=\Om^+_0$, giving rise to a torsion free $G_2$-structure on $M=[0,\varepsilon) \times N$ by (\ref{the3form}). An $SU(3)$-structure satisfying (\ref{halfflatone}), (\ref{halfflattwo}) is called half-flat since exactly half of the $SU(3)$ intrinsic torsion equations vanishes \cite{ChiossiSalamonIntrinsicTorsion} and in \cite{Hitchin01stableforms} Hitchin showed that this condition is indeed preserved under the flow. The above set up has proven to be incredibly useful in constructing cohomogeneity one examples of $G_2$ metrics and $G_2$-instantons starting from half-flat $SU(3)$-structures on homogeneous spaces cf. \cite{Foscolo2018, Hitchin01stableforms, Lotay2016}. Thus, we are naturally led to consider the analogous problem for deformed $G_2$-instantons.

So consider a connection on $M=I_t \times N$ of the form 
\begin{equation}
A=a(t),\label{ansatzconnectionproduct}
\end{equation}
where $a(t)$ denotes a family of connections on $N$.
The curvature form is then given by
\begin{equation}
	F_A=dt\w a'+F_a,
\end{equation}
where $F_a=d_Na$ with $d_N$ denoting the exterior derivative on $N$. A simple computation now shows that
\begin{Prop}\label{reducedproposition}
The connection form $A$ given by (\ref{ansatzconnectionproduct}) defines a deformed $G_2$-instanton on $M$ if and only if the connection form $a(t)$ on $N$ solves
\begin{gather}
\frac{1}{2}a'\w (F_a^2-\om^2)+F_a\w \Om^-=0,\label{reduceddeformG2two}
\end{gather}
and is subject to the constraint
\begin{equation}
\frac{1}{6}	F_a^3 - \frac{1}{2} F_a \w \om^2=0.\label{reduceddeformG2one}
\end{equation}
Moreover, condition (\ref{reduceddeformG2one}) is preserved under the evolution equation  (\ref{reduceddeformG2two}) i.e. if it is holds at some $t_0\in I_t$ then it holds for all $t\in I_t$.
\end{Prop}
If $N$ is a Calabi-Yau $3$-fold then from Hitchin's flow $M=\R \times N$ inherits the natural product structure. If we also assume that $a(t)$ is $t$-invariant so that $A$ is just a connection on $N$ then the pair (\ref{reduceddeformG2two}), (\ref{reduceddeformG2one}) reduces to the usual deformed Hermitian Yang-Mills equations cf. \cite[Lemma 3.2]{Lotay2020}. 

If $N$ is $S^3 \times S^3$ endowed with its homogeneous nearly K\"ahler $SU(3)$-structure then $M=\R^+\times S^3 \times S^3$ corresponds to the Bryant-Salamon $G_2$ cone described in Section \ref{theg2manifoldsection}, see \cite[Sect. 7.2]{Hitchin01stableforms}. 
More generally, one can consider the case when $N=G/H$ is a homogeneous space with a $G$-invariant half-flat $SU(3)$-structure and $A$ is also a $G$-invariant connection; this is the general set up for constructing invariant connections on cohomogeneity one $G_2$ manifolds. Deformed $G_2$-instantons on $M$ are then described by certain curves in the finite dimensional space of closed $G$-invariant $2$-forms on $N$ and these curves are constrained to lie in the codimension one subvariety determined by (\ref{reduceddeformG2one}). In particular, the examples we described in the previous sections arise in this way. \\

\textbf{An example on the Iwasawa manifold.} 
There are by now a large class of explicitly known examples of half-flat $SU(3)$-structures on homogeneous spaces, see for instance \cite{Conti11} for the classification on nilmanifolds. We illustrate Proposition \ref{reducedproposition} by taking $N$ to be the Iwasawa manifold. $N$ can be viewed as a nilmanifold with a left-invariant coframing $e^i$ satisfying $de^i=0$ for $i=1,2,3,4$ and
\[de^5= \om_1,\ \ \ \  de^6 = -\om_2,\]
where $\om_1=e^{12}+e^{34}$, $\om_2=e^{13}-e^{24}$ and $\om_3=e^{14}+e^{23}$. Observe that the latter structure equations exhibit $N$ as a $\mathbb{T}^2$ bundle over $\mathbb{T}^4$. 
A solution to Hitchin's flow on $N$ is then furnished by
\begin{gather*}
	\om=(3t)^{4/3}\om_3+(3t)^{-2/3}e^{56},\\
	\Om^+=(3t)^{1/3} e^5 \w \om_2 + (3t)^{1/3} e^6 \w \om_1,\\
	\Om^-=(3t)^{1/3} e^6 \w \om_2 - (3t)^{1/3} e^5 \w \om_1.
\end{gather*}
This yields the $G_2$ metric 
\begin{equation}
	g_{\vp}=dt^2 + (3t)^{2/3}((e^1)^2+(e^2)^2+(e^3)^2+(e^4)^2) +(3t)^{-2/3}((e^5)^2+(e^6)^2)\label{apostolovsalamonmetric}
\end{equation}
on $M=\R^+ \times N$ which is incomplete as $t\to 0$ but complete as $t\to \infty$ cf. \cite{ChiossiSalamonIntrinsicTorsion}. 
A general left invariant connection form on $M$ is given by
\begin{equation}
	A=\sum_{i=1}^6f_i(t)e^i.
\end{equation}
It is easy to see that (\ref{reduceddeformG2one}) is automatically satisfied. Computing the $G_2$-instanton equation we find that $f_i'(t)=0$ for $i=1,2,3,4$ and
\begin{equation}
	3 t f'_i(t)+2f_i(t)=0,\label{nilodeg2instanton}
\end{equation}
for $i=5,6$. So without loss of generality we can set $f_i(t)=0$ for $i=1,2,3,4$ and solving (\ref{nilodeg2instanton}) we get a family of $G_2$-instantons
\begin{equation}
A(t)=(3t)^{-2/3}(c_5 e^5+c_6 e^6),
\end{equation}
where $c_5,c_6$ are constants. From (\ref{apostolovsalamonmetric}) it easy to see that the latter corresponds to linear combinations of the $1$-forms dual to the Killing vector fields $e_5$ and $e_6$. 

On the other hand, (\ref{reduceddeformG2two}) becomes equivalent to $f_i'(t)=0$ for $i=1,2,3,4$ as before, but now we get a non-linear system of ODEs: 
\begin{equation}
3 t^{4/3} f_i'(t)+2t^{1/3}f_i(t)-3^{-1/3}(f_i(t)^2 f_i'(t)+ f_j(t)^2 f_i'(t))=0,
\end{equation}
where $(i,j)=(5,6), (6,5)$. A family of deformed $G_2$-instantons is then given by 
\begin{equation}
	A(t)={\sqrt{2}\ (3t)^{2/3}}\Big(\frac{c_5 e^5+c_6 e^6}{c_5^2+c_6^2}\Big).
\end{equation}
By contrast to our previous examples we now see that the asymptotic behaviour of the $G_2$-instantons and deformed $G_2$-instantons are drastically different. A rather similar behaviour is also exhibited by the metric (\ref{apostolovsalamonmetric}) since while the size of  $\mathbb{T}^2$ fibres decays as $t \to \infty$, the size of the base $\mathbb{T}^4$ grows.  
 
Since there are plenty of explicit examples of local $G_2$ metrics arising from nilmanifolds, we expect that a similar construction as above will yield many local examples of deformed $G_2$-instantons. 
Constructing examples on complete $G_2$ manifolds on the other hand appear to be a much harder problem.

Since there is an intricate link between instantons and calibrated submanifolds cf. \cite{Chen2021, LeeLeung2009, Tian2000}, next we describe some new examples in our context.

\section{A remark on associative fibration}\label{calibratedsubmanifolds}

Given a $G_2$ manifold $M$ admitting an associative semi-flat $\mathbb{T}^3$ fibration, in \cite{LeeLeung2009} Lee-Leung argue that the Fourier-Mukai transform should interchange the associative geometry of $M$ with the coassociative geometry of the moduli space of associative cycles on $M$. So this naturally motivates us to search for associative fibrations in our set up.
While on smooth compact manifolds this is not possible since the moduli space of associatives has expected dimension 0, in the non-compact setting there are no such restrictions. Indeed the $G_2$-structures constructed by Bryant-Salamon on the anti-self-dual bundles of $S^4$ and $\C\mathbb{P}^2$ \cite{Bryant1989} are fibred by associative $\R^3$. 

For the $G_2$-structures described in Section \ref{theg2manifoldsection} it is easy to see that the $\R^4$ fibres of the spinor bundle of $S^3$ define a coassociative fibration and that the zero section $S^3$ corresponds to an associative submanifold, see \cite[Sect. 4]{KarigiannisLotay2021} for a detailed study in the Bryant-Salamon case. To the best of our knowledge, there seems to be no reference in the literature pointing out that there is also an associative foliation of $\R^4\times S^3$. The tangent space of the latter leaves are determined by the vector fields $\langle T^+_1,T^-_1,\partial_r \rangle$. Indeed this distribution is involutive since 
$$[T^+_1,T^-_1]=[T^+_1,\partial_r]=[T^-_1,\partial_r]=0$$
and from (\ref{anstazform}) and (\ref{the3form}) we see that the leaves are calibrated by $\vp$. 
It is also not hard to see that these associatives are all diffeomorphic to $\R^2 \times S^1$. However, the latter does not define a global fibration since an $S^2$ family of the leaves will intersect in the zero section $S^3$ in an $S^1$.  
Since this foliation is parametrised by $S^2 \times S^2$, on the complement of the zero section  we get the associative fibration $$\R^+\times \mathbb{T}^2 \hookrightarrow (\R^4\setminus 0) \times S^3 \to S^2 \times S^2.$$ 
From the description in Section \ref{theg2manifoldsection}, it is easy to see that the fibre metrics in the Bryant-Salamon and BGGG cases are respectively given by:
\begin{gather}
g=\frac{1}{1-c^3r^{-3}}dr^2+\frac{r^2(1-c^3r^{-3})}{9}d\theta_1^2+\frac{r^2}{3}d\theta_2^2,\\
g=\frac{(r-3/4)(r+3/4)}{(r-9/4)(r+9/4)}dr^2+\frac{(r-9/4)(r+9/4)}{(r-3/4)(r+3/4)}d\theta_1^2+\frac{4r^2}{9}d\theta_2^2,
\end{gather}
where $\theta_i$ denote a local coordinate on the $S^1$ factors. In the Bryant-Salamon case observe that the metrics are asymptotically conical; these can be viewed as smoothings of the conical examples found by Lotay on $\R^7$ in \cite{Lotay2006}. On the other hand, in the BGGG case the metrics are asymptotic to a flat product metric on $\R^2 \times S^1$. It is rather interesting that the geometry of the associatives encapsulates the behaviour of the respective $G_2$ metrics.


\bibliography{biblioG}
\bibliographystyle{plain}

\end{document}